\documentclass[10pt,a4paper]{article}
\usepackage[latin1]{inputenc}
\usepackage{amsmath}
\usepackage{amsfonts}
\usepackage{amssymb}
\usepackage{graphicx}
\usepackage{amsthm} 	
\usepackage{hyperref}
\usepackage{todonotes}
\usetikzlibrary{backgrounds}	
\usepackage[toc]{appendix}

\usepackage{graphicx,xcolor}	

\usepackage{xcolor}
\pagecolor{white}

\usepackage{esint}				

\usepackage[a4paper,top=2.1cm,bottom=2.60cm,left=3.8cm,right=3.8cm]{geometry} 

\allowdisplaybreaks

\newtheorem{theorem}{Theorem}[section]

\newtheorem{corollary}[theorem]{Corollary}

\newtheorem{lemma}[theorem]{Lemma}

\title{Shape analyticity and singular perturbations for layer potential operators}
\author{Matteo Dalla Riva  \thanks{Dipartimento di Ingegneria, Universit\`a degli Studi di Palermo, Viale delle Scienze, Ed. 8, 90128 Palermo, Italy.} ,  Paolo Luzzini \thanks{Dipartimento di Matematica ``Tullio Levi-Civita'', Universit\`a degli Studi di Padova, Via Trieste 63, Padova 35121, Italy.} , Paolo Musolino\thanks{Dipartimento di Scienze Molecolari e Nanosistemi, Universit\`a Ca' Foscari Venezia, via Torino 155, 30172 Venezia Mestre, Italy}}

\date{20220303PL\_shapeint}

\begin{document}

\maketitle

\noindent

{\bf Abstract:}   We study the effect of regular and singular domain perturbations on layer potential operators for the Laplace equation. %
First, we consider layer potentials supported on a diffeomorphic image $\phi(\partial\Omega)$ of a reference set $\partial\Omega$ and we present some real analyticity results for the dependence upon the map $\phi$. Then we introduce a perforated domain $\Omega(\epsilon)$ with a small hole of size $\epsilon$ and we compute power series expansions that describe the layer potentials on $\partial\Omega(\epsilon)$ when the parameter $\epsilon$ approximates the degenerate value $\epsilon=0$.

\vspace{9pt}

\noindent
{\bf Keywords:}  single layer potential, double layer potential, Laplace operator, domain perturbation, shape sensitivity analysis, perforated domain, asymptotic behavior, special nonlinear operators.
\vspace{9pt}

\noindent   
{{\bf 2020 Mathematics Subject Classification:}}  31B10; 35J05; 47H30; 35J25; 45A05.

\section{Introduction}\label{sec:intro}

Potential theory is a valuable tool to analyze boundary value problems for elliptic differential equations and systems, both to deduce theoretical results and to perform numerical computations. Indeed, layer potentials can be used to convert boundary value problems into systems of integral equations that are often easier to study than the original problems. In recent times, potential theoretic techniques have been successfully employed to analyze boundary value problems on perturbed domains. In view of this application, it is important to understand what happens to the layer potentials when we perturb the support of integration. In this paper we look at this problem in the terms of the following question: what is the regularity of the maps that take the perturbation parameters to the corresponding layer potential operators?

To try to give an answer, we will consider the layer potentials related to the Laplace equation and we  will study two different kind of perturbations,  one that we call ``regular," because we don't have loss of regularity in the perturbed sets,  and one that we call ``singular,'' because we do have some kind of loss of regularity in the perturbed sets.  More specifically, as an example of a regular perturbation we will have layer potentials supported on a set $\phi(\partial\Omega)$ that is a diffeormorphic image of the boundary $\partial\Omega$ of a reference set $\Omega$. In this case the perturbation parameter is the map $\phi$ and our goal is to understand the regularity of the map that takes $\phi$, which we think as an element of a suitable Banach space of functions, to the corresponding layer potential operators, which we think as elements of suitable operator spaces. Instead, to make an example of a singular perturbation, we will analyze layer potentials supported on a set $\partial\Omega(\epsilon)$ with $\Omega(\epsilon)$   obtained removing from a fixed domain $\Omega$ an interior portion of size   $\epsilon>0$. This perturbation is ``singular'' because for $\epsilon=0$ the  set $\Omega(\epsilon)$   loses regularity on account of a removed point in its iterior. Also in this second case our goal is to understand the regularity of the map that takes the perturbation parameter --in this case $\epsilon$-- to the layer potential operators. In particular, we will focus on the situation where $\epsilon$ varies in a neighborhood of zero.

The interest for the regularity of this kind of maps can be motivated by the applications that they have in the framework of inverse scattering problems. For example, in the works \cite{Po94, Po96a, Po96b} of Potthast, we may find   Fr\'echet differentiability  results for certain layer potentials related to the Helmholtz equation (see also Haddar and Kress \cite{HaKr04} for a further application). Charalambopoulos \cite{Ch95} obtained similar results, but for the layer potentials related to the elastic scattering problem. In the sense introduced above,  the perturbations considered by Potthast and Charalambopoulos are of regular type: they consider a reference set of class $C^2$ that is perturbed into a new set that remains of class $C^2$. The regularity of the sets allows them to keep the analysis in the context of Schauder spaces. The case of Lipshitz domains, instead, was analyzed by Costabel and Le Lou\"er in \cite{CoLe12a, CoLe12b, Le12} in the framework of Sobolev spaces. The family of layer potentials considered by Costabel and Le Lou\"er is quite general and includes the usual boundary integral operators occurring in time-harmonic potential theory.

Now, all the papers listed in the previous paragraph deal with differentiability properties 
and, indeed, regularity results that go beyond the differentiability seem to be much rarer in literature. There are some examples though.  For instance, the recent work on the ``shape holomorphy''  by Henr\'iquez and Schwab \cite{HeSc21}, where they consider the layer potential operators supported on a $C^2$ Jordan curve in $\mathbb{R}^2$. In Henr\'iquez and Schwab's paper a suitable parametrization of the Jordan curve plays the role of the (regular) perturbation parameter, which they think as an element in a complex Banach space,   and, among other results, they show that the Calder\'on projector of the two-dimensional Laplacian is an holomorphic map of such parametrization. The idea of ``shape holomorphy''  was previously introduced in the papers by Jerez-Hanckes, Schwab, and Zech \cite{JeScZe17}, dedicated to the electromagnetic wave scattering problem, and by Cohen, Schwab, and Zech \cite{CoScZe18}, about the stationary Navier-Stokes equations.

Also the present paper's goal is to discuss regularity properties beyond the differentiability. More specifically, our aim is to prove real analyticity results. So, for example, in the first part of the paper, where we consider layer potentials supported on a diffeomorphic image $\phi(\partial \Omega)$ of a reference set $\partial\Omega$,
we show that the map that takes $\phi$ to the corresponding layer potential operators is real analytic. The results of this first part are a direct consequence of the work of Lanza de Cristoforis and Rossi in \cite{LaRo04, LaRo08} and they can be compared with the holomorphy results proven by Henr\'iquez and Schwab in \cite{HeSc21}. Indeed, real analytic maps can be extended to holomorphic maps between reasonable complexifications of the underlying Banach spaces  (see, e.g., the monograph of H\`ayes and Johanis \cite{HaJo14} and the references therein, see also the paragraph after Corollary \ref{cor:anCal}). Although the restriction to the two-dimensional case might not be essential in Henr\'iquez and Schwab paper, we also remark that here we consider all dimensions $n\ge 2$.
 
In addition to the above mentioned papers \cite{LaRo04, LaRo08}, which are dedicated to the layer potentials for the Laplace and  Helmholtz equations,  Lanza de Cristoforis and collaborators have extensively studied this kind of problems in many different directions.  For example, in \cite{DaLa10a} the authors considered  a family of fundamental solutions of second order constant coefficient differential operators  and proved that the corresponding layer potentials depend   real analytically  jointly on the parametrization of the support, the density, and the coefficients of the  operators. We also mention  \cite{Da08}, where a similar result was obtained for higher order operators, and \cite{LaMu11}, for the case of periodic layer potentials.  Moreover, analyticity properties of the layer potentials have been exploited by the authors to analyze the shape dependence of physical quantities arising in fluid mechanics, material sciences, and scattering theory (see \cite{DaLuMu21, DaLuMuPu21, LuMu20, LuMuPu19}).

So, we might say that, as long as it concerns the problem proposed in this paper, regular perturbations are the subject of several works. Singular perturbations, instead,  are widely studied in relation to boundary value problems and inverse problems (see, for example, Ammari and Kang \cite{AmKa07}, Ammari, Kang, and Lee \cite{AmKaLe09},  Maz'ya, Movchan, and Nieves \cite{MaMoNi13}, Maz'ya, Nazarov, and Plamenevskii \cite{MaEtAl00a, MaEtAl00b}, and the references therein), but  seem to be far less studied in relation with the regularity of the layer potential operators maps. An exception is the work carried out by Lanza de Cristoforis and his collaborators with the development of the  so called  ``functional analytic approach'' (see the seminal papers \cite{La02, La05, La07}, see also \cite{DaLaMu21} and the references therein). To illustrate an application of the functional analytic approach we consider a domain $\Omega(\epsilon)$ with a hole of size $\epsilon$. We first show that we can write the layer potential operators in terms of  real analytic maps of $\epsilon$, which are defined in an open neighborhood of $\epsilon=0$, and of  continuous elementary functions of $\epsilon$, which might be not smooth, or even  singular for $\epsilon=0$. Then we focus on the analytic maps and we show how we can compute explicitly the coefficients of the corresponding power series expansions. The technique to compute such coefficients is inspired by the work in \cite{DaMuRo15}, where the computation was carried out in the case of a Dirichlet problem in a domain with a small hole (we incidentally note that a recent paper \cite{FeAm21} by Feppon and Ammari presents a result comparable with that of \cite{DaMuRo15}).

The paper is organized as follows. In Section \ref{s:prel} we introduce some notation, mainly related to layer potentials. In Section \ref{s:reg} we recall the results of Lanza de Cristoforis and Rossi \cite{LaRo04,LaRo08} on regular domain perturbations and we deduce some other analyticity results. We also include a paragraph where we discuss the relation between real analyticity and holomorphy. In Section \ref{s:sin},  we consider singular domain perturbations and, after having deduced representations in terms of known elementary functions and real analytic maps, we show an explicit and  constructive way to compute all the coefficients of the corresponding power series expansions.

\section{Layer potentials for the Laplace equation}\label{s:prel}
 
 In this section, we introduce the layer potentials (and associated operators) for the Laplace equation. In order to do so, we fix
 \[
 n \in \mathbb{N} \setminus \{0,1\}\, 
 \]
and we take
\[
\begin{split}
&\text{$\alpha \in \mathopen]0,1[$ and a bounded open connected subset $\tilde{\Omega}$ of $\mathbb{R}^{n}$ of  class  $C^{1,\alpha}$}.
\end{split}
\]
 For the definition of sets and functions of the Schauder class $C^{j,\alpha}$ ($j \in \mathbb{N}$)  we refer, e.g., to Gilbarg and
Trudinger~\cite{GiTr83}.

Let
$G_{n}$ be the function from ${\mathbb{R}}^{n}\setminus\{0\}$ to ${\mathbb{R}}$ defined by
\[
G_{n}(x)\equiv
\left\{
\begin{array}{lll}
-\frac{1}{s_{2}}\log |x| \qquad &   \forall x\in
{\mathbb{R}}^{n}\setminus\{0\},\quad & {\mathrm{if}}\ n=2\,,
\\
\frac{1}{(n-2)s_{n}}|x|^{2-n}\qquad &   \forall x\in
{\mathbb{R}}^{n}\setminus\{0\},\quad & {\mathrm{if}}\ n\geq3\,,
\end{array}
\right.
\]
where $s_n$ denotes the $(n-1)$-dimensional measure of the unit sphere in $\mathbb{R}^n$. The function
$G_{n}$ is well-known to be a
fundamental solution of $-\Delta\equiv -\sum_{i=1}^n\partial_{x_j}^2$. \par

We now introduce the single layer potential. If $\mu\in C^{0}(\partial\tilde{\Omega})$, we set 
\begin{equation*}
\mathcal{S}_{\tilde{\Omega}}[\mu](x)\equiv
\int_{\partial\tilde{\Omega}}G_n(x-y)\mu(y)\,d\sigma_{y}
\qquad\forall x\in {\mathbb{R}}^{n}\,,
\end{equation*}
 where  $d\sigma$ denotes the area element of a $(n-1)$-dimensional manifold imbedded in ${\mathbb{R}}^{n}$. As is well-known, if $\mu\in C^{0}(\partial{\tilde{\Omega}})$, then $\mathcal{S}_{\tilde{\Omega}}[\mu]$ is continuous in  ${\mathbb{R}}^{n}$. Moreover, if $\mu\in C^{0,\alpha}(\partial\tilde{\Omega})$, then the function  
 $\mathcal{S}^{\mathrm{int}}_{\tilde{\Omega}}[\mu]\equiv \mathcal{S}_{\tilde{\Omega}}[\mu]_{|\overline{\tilde{\Omega}}}$ belongs to $C^{1,\alpha}(\overline{\tilde{\Omega}})$, and the function 
$\mathcal{S}^{\mathrm{ext}}_{\tilde{\Omega}}[\mu]\equiv \mathcal{S}_{\tilde{\Omega}}[\mu]_{|\mathbb{R}^n \setminus \tilde{\Omega}}$ belongs to $C^{1,\alpha}_{\mathrm{loc}}
(\mathbb{R}^n \setminus \tilde{\Omega})$. As usual,  $\overline{A}$ denotes  the  closure of a set $A$.

Similarly, we introduce the double layer potential. If $\psi\in C^{0}(\partial\tilde{\Omega})$, we set
\begin{equation*}
\mathcal{D}_{\tilde{\Omega}}[\psi](x)\equiv
-\int_{\partial\tilde{\Omega}}\nu_{\tilde{\Omega}}(y) \cdot \nabla  G_n(x-y)\psi(y)\,d\sigma_{y}
\qquad\forall x\in {\mathbb{R}}^{n}\,,
\end{equation*}
where $\nu_{\tilde{\Omega}}$ denotes the outer unit normal to $\partial{\tilde{\Omega}}$  and the symbol ``$\cdot$'' denotes the scalar product in $\mathbb{R}^n$.   As is well known, if $\psi\in C^{1,\alpha}(\partial\tilde{\Omega})$ the restriction $\mathcal{D}_{\tilde{\Omega}}[\psi]_{|\tilde{\Omega}}$ extends to a function $\mathcal{D}^{\mathrm{int}}_{\tilde{\Omega}}[\psi]$  in  $C^{1,\alpha}(\overline{\tilde{\Omega}})$ and  the restriction $\mathcal{D}_{\tilde{\Omega}}[\psi]_{|\mathbb{R}^n\setminus\overline{\tilde{\Omega}}}$ extends to a function $\mathcal{D}^{\mathrm{ext}}_{\tilde{\Omega}}[\psi]$  in  $C^{1,\alpha}_{\mathrm{loc}}(\mathbb{R}^n\setminus\tilde{\Omega})$. We observe that the symbols $\mathcal{D}^{\mathrm{int}}_{\tilde{\Omega}}[\psi]$ and $\mathcal{D}^{\mathrm{ext}}_{\tilde{\Omega}}[\psi]$ denote the extensions of the restrictions of the double layer potential to the closure of the interior and of the exterior of $\tilde{\Omega}$, respectively.

Next, we introduce two operators  associated with the boundary trace of the double layer potential and of the normal derivative of the single layer potential.  Let 
\begin{equation}\label{Kdef}
\mathcal{K}_{\tilde{\Omega}}[\psi](x)\equiv \mathcal{D}_{\tilde{\Omega}}[\psi]_{|\partial\tilde\Omega}(x) = - \int_{\partial\tilde{\Omega}}\nu_{\tilde{\Omega}}(y) \cdot \nabla G_n(x-y)\psi(y)\,d\sigma_{y}\qquad\forall x\in\partial\tilde{\Omega}\,,
\end{equation}
 for all $\psi\in C^{1,\alpha}(\partial\tilde{\Omega})$, and
\begin{equation}\label{K'def}
\mathcal{K}'_{\tilde{\Omega}}[\mu](x)\equiv \int_{\partial\tilde{\Omega}}\nu_{\tilde{\Omega}}(x) \cdot \nabla G_n(x-y)\mu(y)\,d\sigma_{y}\qquad\forall x\in\partial\tilde{\Omega}\,,
\end{equation} 
for all $\mu\in C^{0,\alpha}(\partial{\tilde{\Omega}})$. As it is well-known from classical potential theory,  $\mathcal{K}_{\tilde{\Omega}}$ is a compact operator  from   $C^{1,\alpha}(\partial{\tilde{\Omega}})$  to itself and $\mathcal{K}'_{\tilde{\Omega}}$   is a compact operator  from  $C^{0,\alpha}(\partial{\tilde{\Omega}})$ to itself (see Schauder \cite{Sc31, Sc32}). Also, the operators $\mathcal{K}_{\tilde{\Omega}}$ and $\mathcal{K}'_{\tilde{\Omega}}$ are adjoint one to the other with respect to the duality on $C^{1,\alpha}(\partial{\tilde{\Omega}})\times C^{0,\alpha}(\partial{\tilde{\Omega}})$ induced by the inner product of the Lebesgue space $L^2(\partial{\tilde{\Omega}})$ (cf.,  e.g., Kress \cite[Chap.~4]{Kr14}). Moreover,  the following jump formulas, describing the boundary behavior of the layer potentials with the corresponding boundary operators, hold. 
\begin{align*}
\mathcal{D}^{\mathrm{int}}_{{\tilde{\Omega}}}[\psi]_{|\partial{\tilde{\Omega}}}&=-\frac{1}{2}\psi+\mathcal{K}_{\tilde{\Omega}}[\psi]&\forall\psi\in C^{1,\alpha}(\partial{\tilde{\Omega}})\,,\\
 \mathcal{D}^{\mathrm{ext}}_{{\tilde{\Omega}}}[\psi]_{|\partial{\tilde{\Omega}}}&=\frac{1}{2}\psi+\mathcal{K}_{\tilde{\Omega}}[\psi]&\forall\psi\in C^{1,\alpha}(\partial{\tilde{\Omega}})\,,\\
\nu_{\tilde{\Omega}}\cdot\nabla \mathcal{S}^{\mathrm{int}}_{{\tilde{\Omega}}}[\mu]_{|\partial{\tilde{\Omega}}}&=\frac{1}{2}\mu+\mathcal{K}'_{\tilde{\Omega}}[\mu] &\forall\mu\in C^{0,\alpha}(\partial{\tilde{\Omega}})\, ,\\
 \nu_{\tilde{\Omega}}\cdot\nabla \mathcal{S}^{\mathrm{ext}}_{{\tilde{\Omega}}}[\mu]_{|\partial{\tilde{\Omega}}}&=-\frac{1}{2}\mu+\mathcal{K}'_{\tilde{\Omega}}[\mu] &\forall\mu\in C^{0,\alpha}(\partial{\tilde{\Omega}})\, 
\end{align*}
(see,  e.g., Folland \cite[Chap.~3]{Fo95}).

 Finally, we also set
\begin{equation}\label{Vdef}
\mathcal{V}_{\tilde{\Omega}}[\mu](x)\equiv\mathcal{S}_{\tilde{\Omega}}[\mu](x) \qquad \forall x \in \partial \tilde{\Omega}\, ,
\end{equation}
for all $\mu\in C^{0,\alpha}(\partial{\tilde{\Omega}})$, and
\begin{equation}\label{Wdef}
\mathcal{W}_{\tilde{\Omega}}[\psi](x)\equiv- \nu_{\tilde{\Omega}}(x)\cdot \nabla \mathcal{D}^{\mathrm{ext}}_{{\tilde{\Omega}}}[\psi](x)= -\nu_{\tilde{\Omega}}(x)\cdot \nabla \mathcal{D}^{\mathrm{int}}_{{\tilde{\Omega}}}[\psi](x)\qquad \forall x \in \partial \tilde{\Omega}\, ,
\end{equation}
for all $\psi\in C^{1,\alpha}(\partial{\tilde{\Omega}})$ (see, e.g., \cite[Thm. 4.31 (iii)]{DaLaMu21}). Clearly, $\mathcal{V}_{\tilde{\Omega}}[\mu] \in C^{1,\alpha}(\partial{\tilde{\Omega}})$ for all $\mu\in C^{0,\alpha}(\partial{\tilde{\Omega}})$ and $\mathcal{W}_{\tilde{\Omega}}[\psi] \in C^{0,\alpha}(\partial{\tilde{\Omega}})$ for all $\psi\in C^{1,\alpha}(\partial{\tilde{\Omega}})$.


\section{Regular perturbations and shape analyticity}\label{s:reg}

In this section we consider layer potentials supported on the diffeomorphic image of a reference set. We show some results of Lanza de Cristoforis and Rossi \cite{LaRo04, LaRo08} on the real analyticity of the maps that take the parametrization to the corresponding layer potentials. From these results we  deduce some analyticity results for the corresponding operators.

We now introduce the geometry of the problem. We fix
\begin{equation}\label{Omega_def}
\begin{split}
&\text{$\alpha \in \mathopen]0,1[$ and a bounded open connected subset $\Omega$ of $\mathbb{R}^n$ of  class  $C^{1,\alpha}$}
\\
&\text{\hspace{2.5cm}such that $\mathbb{R}^n\setminus\overline{\Omega}$ is  connected.}
\end{split}
\end{equation}

To consider shape perturbations of layer potential operators, we take the set $\Omega$ of \eqref{Omega_def} as a reference set. Then we  introduce a specific class $\mathcal{A}^{1,\alpha}_{\partial \Omega}$ of $C^{1,\alpha}$-diffeomorphisms from $\partial\Omega$ to $\mathbb{R}^n$: $\mathcal{A}^{1,\alpha}_{\partial \Omega}$ is the set of functions of class $C^{1,\alpha}(\partial\Omega, \mathbb{R}^n)$ that are injective and have injective differential at all points  of $\partial\Omega$. By Lanza de Cristoforis and Rossi \cite[Lem. 2.2, p. 197]{LaRo08}  
and \cite[Lem. 2.5, p. 143]{LaRo04},  we can see that $\mathcal{A}^{1,\alpha}_{\partial \Omega}$ is open 
in $ C^{1,\alpha}(\partial\Omega, \mathbb{R}^n)$. Moreover, for all $\phi \in \mathcal{A}^{1,\alpha}_{\partial \Omega}$ the Jordan-Leray separation theorem ensures that 
$\mathbb{R}^n\setminus \phi(\partial \Omega)$  has exactly two open connected components (see, e.g., Deimling \cite[Thm.  5.2, p. 26]{De85}  and \cite[\S A.4]{DaLaMu21}). We denote  by $\mathbb{I}[\phi]$ the bounded connected component of $\mathbb{R}^n\setminus \phi(\partial \Omega)$ and by
$\mathbb{E}[\phi]$ the  unbounded one. Then, we have
$\mathbb{E}[\phi] =\mathbb{R}^n \setminus \overline{\mathbb{I}[\phi]}$ and  $\overline{\mathbb{E}[\phi]} =\mathbb{R}^n \setminus \mathbb{I}[\phi]$ (see Figure \ref{fig:DiffeoOm}).

\begin{figure}[htb]
\centering
\includegraphics[width=4.2in]{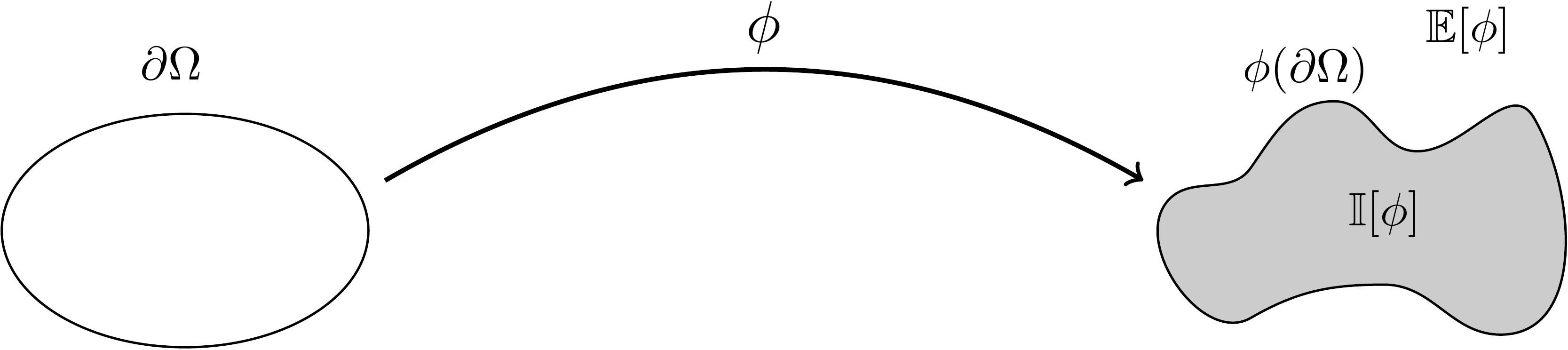}\\
\begin{center}
\caption{{\it The diffeomorphism  $\phi \in \mathcal{A}^{1,\alpha}_{\partial \Omega}$ and the $\phi$-dependent sets $\phi(\partial\Omega)$, $\mathbb{I}[\phi]$ and $\mathbb{E}[\phi]$.}}
		\label{fig:DiffeoOm}
\end{center}
\end{figure}

 We will think at the diffeomorphism $\phi$ as a point in the Banach space $C^{1,\alpha}(\partial\Omega, \mathbb{R}^n)$ and we want to see that the maps that take $\phi\in \mathcal{A}^{1,\alpha}_{\partial \Omega}\subseteq C^{1,\alpha}(\partial\Omega, \mathbb{R}^n)$ to the operators $\mathcal{V}_{\mathbb{I}[\phi]}$, $\mathcal{K}_{\mathbb{I}[\phi]}$, $\mathcal{K}'_{\mathbb{I}[\phi]}$, and $\mathcal{W}_{\mathbb{I}[\phi]}$ are, in a sense, real analytic. We observe, however, that these operators are elements of spaces that depend on $\phi$. For example, $\mathcal{V}_{\mathbb{I}[\phi]}$ belongs to 
\[
\mathcal{L}(C^{0,\alpha}(\phi(\partial\Omega)), C^{1,\alpha}(\phi(\partial\Omega)))\,.
\]
So, to have real analytic maps between fixed Banach spaces we  ``pull-back'' the operators to the reference set $\partial\Omega$. For example, for a diffeomorphism $\phi\in \mathcal{A}^{1,\alpha}_{\partial \Omega}$, we denote by $\mathcal{V}_\phi$ the operator that takes a density function $\mu\in C^{0,\alpha}(\partial\Omega)$ to $\mathcal{V}_{\mathbb{I}[\phi]}[\mu\circ\phi^{(-1)}]\circ\phi$. Namely, we set 
\[
\mathcal{V}_{\phi} [\mu]\equiv \mathcal{V}_{\mathbb{I}[\phi]}[\mu \circ \phi^{(-1)}] \circ \phi \qquad \forall \mu \in C^{0,\alpha}(\partial\Omega)\, .
\]
Then we see that $\mathcal{V}_{\phi}$ is an element of the space
\[
\mathcal{L}(C^{0,\alpha}(\partial\Omega), C^{1,\alpha}(\partial\Omega))\,,
\]
which does not depend on $\phi$, and it makes sense to ask if the map $\phi\mapsto\mathcal{V}_\phi$ is real analytic. (We refer, e.g., to Deimling \cite[\S15]{De85} for the definition of real analytic maps between Banach spaces.) Similarly, we denote by $\mathcal{K}_{\phi}$ the element of $\mathcal{L}(C^{1,\alpha}(\partial\Omega), C^{1,\alpha}(\partial\Omega))$ such that
\[
\mathcal{K}_{\phi} [\psi]\equiv \mathcal{K}_{\mathbb{I}[\phi]}[\psi \circ \phi^{(-1)}] \circ \phi \qquad \forall \psi \in C^{1,\alpha}(\partial\Omega)\,,
\]
we denote by $\mathcal{K}'_{\phi}$ the element of $\mathcal{L}(C^{0,\alpha}(\partial\Omega), C^{0,\alpha}(\partial\Omega))$ defined by
\[
\mathcal{K}'_{\phi} [\mu]\equiv \mathcal{K}'_{\mathbb{I}[\phi]}[\mu \circ \phi^{(-1)}] \circ \phi \qquad \forall \mu \in C^{0,\alpha}(\partial\Omega)\,,
\]
and by $\mathcal{W}_{\phi}$ the element of $\mathcal{L}(C^{1,\alpha}(\partial\Omega), C^{0,\alpha}(\partial\Omega))$ defined by
\[
\mathcal{W}_{\phi} [\psi]\equiv \mathcal{W}_{\mathbb{I}[\phi]}[\psi \circ \phi^{(-1)}] \circ \phi \qquad \forall \psi \in C^{1,\alpha}(\partial\Omega)\, .
\]

In the following Lemma \ref{lem:anSLP} we present some results from  Lanza de Cristoforis and Rossi \cite{LaRo04, LaRo08}.

\begin{lemma}\label{lem:anSLP}
Let $\alpha$, $\Omega$ be as in \eqref{Omega_def}.  
Then the following statements hold.  
\begin{itemize}
 \item[(i)] The map from $\mathcal{A}^{1,\alpha}_{\partial \Omega}\times C^{0,\alpha}(\partial\Omega)$ to  $C^{1,\alpha}(\partial\Omega)$ that takes a pair $(\phi,\mu)$ to the function $\mathcal{V}_{\mathbb{I}[\phi]}[\mu \circ \phi^{(-1)}] \circ \phi$ is real analytic.
 \item[(ii)] The map from $\mathcal{A}^{1,\alpha}_{\partial \Omega} \times C^{1,\alpha}(\partial\Omega)$ to  $C^{1,\alpha}(\partial\Omega)$ that takes a pair $(\phi,\psi)$ to the function $\mathcal{K}_{\mathbb{I}[\phi]}[\psi \circ \phi^{(-1)}] \circ \phi$ is real analytic.
\item[(iii)] The map from $\mathcal{A}^{1,\alpha}_{\partial \Omega} \times C^{0,\alpha}(\partial\Omega)$ to  $C^{0,\alpha}(\partial\Omega)$ that takes a pair $(\phi,\mu)$ to the function $\mathcal{K}'_{\mathbb{I}[\phi]}[\mu \circ \phi^{(-1)}] \circ \phi$ is real analytic.
\item[(iv)] The map from $\mathcal{A}^{1,\alpha}_{\partial \Omega} \times C^{1,\alpha}(\partial\Omega)$ to  $C^{0,\alpha}(\partial\Omega)$ that takes a pair $(\phi,\psi)$ to the function $\mathcal{W}_{\mathbb{I}[\phi]}[\psi \circ \phi^{(-1)}] \circ \phi$ is real analytic.
\end{itemize}
\end{lemma}

By Lemma \ref{lem:anSLP} we deduce the validity of the following theorem, where we show that the  operators $\mathcal{V}_{\phi}$, $\mathcal{K}_{\phi}$, $\mathcal{K}'_{\phi}$, and $\mathcal{W}_{\phi}$ depend real analytically on $\phi$.

 \begin{theorem}\label{thm:anlpop}
Let $\alpha$, $\Omega$ be as in \eqref{Omega_def}.  
Then the following statements hold.  
\begin{itemize}
 \item[(i)] The map from $\mathcal{A}^{1,\alpha}_{ \partial \Omega}$ to  $\mathcal{L}(C^{0,\alpha}(\partial\Omega), C^{1,\alpha}(\partial\Omega))$ that takes $\phi$ to  $\mathcal{V}_{\phi}$ is real analytic.
 \item[(ii)] The map from $\mathcal{A}^{1,\alpha}_{\partial \Omega}$ to  $\mathcal{L}(C^{1,\alpha}(\partial\Omega), C^{1,\alpha}(\partial\Omega))$ that takes $\phi$ to   $\mathcal{K}_{\phi}$ is real analytic.
\item[(iii)] The map from $\mathcal{A}^{1,\alpha}_{\partial \Omega}$ to  $\mathcal{L}(C^{0,\alpha}(\partial\Omega), C^{0,\alpha}(\partial\Omega))$ that takes $\phi$ to   $\mathcal{K}'_{\phi}$ is real analytic.
\item[(iv)] The map from $\mathcal{A}^{1,\alpha}_{\partial \Omega}$ to  $\mathcal{L}(C^{1,\alpha}(\partial\Omega), C^{0,\alpha}(\partial\Omega))$ that takes $\phi$ to   $\mathcal{W}_{\phi}$ is real analytic.
\end{itemize}
\end{theorem}

\begin{proof}
We prove only statement (i). The proof of statements (ii)-(iv) can be effected similarly and is accordingly left to the reader. By Lemma \ref{lem:anSLP} the map 
     \[
\mathcal{A}^{1,\alpha}_{\partial \Omega}\times  C^{0,\alpha}(\partial\Omega)\ni     (\phi,\mu)\mapsto \mathcal{V}^\sharp(\phi,\mu)\equiv \mathcal{V}_{\mathbb{I}[\phi]}[\mu \circ \phi^{(-1)}] \circ \phi \in C^{1,\alpha}(\partial\Omega)
     \]
      is real analytic.  Since $\mathcal{V}^\sharp$ is linear and continuous with respect to the variable $\mu$, we have
\[
\mathcal{V}_{\phi^\sharp} =d_\mu \mathcal{V}^\sharp(\phi^\sharp,\mu^\sharp)
\qquad\forall (\phi^\sharp, \mu^\sharp)\in \mathcal{A}^{1,\alpha}_{\partial \Omega}\times  C^{0,\alpha}(\partial\Omega)\,.
\]
Since the right-hand side equals a partial Fr\'{e}chet differential of a  map which is real analytic by Lemma \ref{lem:anSLP} (i), the right-hand side is analytic on $(\phi^\sharp,\mu^\sharp)$. Hence $(\phi^\sharp,\mu^\sharp) \mapsto \mathcal{V}_{\phi^\sharp}$ is real analytic on $\mathcal{A}^{1,\alpha}_{\partial \Omega} \times C^{0,\alpha}(\partial\Omega)$ and, since it does not depend on $\mu^\sharp$, we conclude that it is real analytic on  $\mathcal{A}^{1,\alpha}_{\partial \Omega}$. 
\end{proof}

As in Henr\'iquez and Schwab \cite{HeSc21}, we now introduce the element $\mathcal{C}_{\phi}$ of $\mathcal{L}(C^{1,\alpha}(\partial\Omega) \times C^{0,\alpha}(\partial\Omega),C^{1,\alpha}(\partial\Omega) \times C^{0,\alpha}(\partial\Omega))$ defined by
\[
\mathcal{C}_{\phi}\equiv
\begin{pmatrix}
\frac{1}{2}I -\mathcal{K}_\phi & \mathcal{V}_\phi\\
\mathcal{W}_\phi & \frac{1}{2}I +\mathcal{K}'_\phi\\
\end{pmatrix}
\]
 for all $\phi \in \mathcal{A}^{1,\alpha}_{\partial \Omega}$. In other words, if  $\phi \in \mathcal{A}^{1,\alpha}_{\partial \Omega}$ and $(\psi,\mu)\in C^{1,\alpha}(\partial\Omega) \times C^{0,\alpha}(\partial\Omega)$, then
 \[
\mathcal{C}_{\phi}[\psi,\mu]\equiv \Bigg(\frac{1}{2}\psi -\mathcal{K}_\phi[\psi] + \mathcal{V}_\phi[\mu],\mathcal{W}_\phi[\psi] + \frac{1}{2}\mu +\mathcal{K}'_\phi[\mu]\Bigg)\, . 
 \] 
 The operator $\mathcal{C}_{\phi}$ is called Calder\'on projector.  In Henr\'iquez and Schwab \cite{HeSc21}, it has been proved the shape holomorphy of the Calder\'on projector for the Laplacian in $\mathbb{R}^2$. Here, by Theorem \ref{thm:anlpop}, we  immediately deduce the validity of the following corollary, where we show that the map that takes $\phi$ to  $\mathcal{C}_\phi$  is real analytic (for the case of arbitrary dimension $n\geq 2$).

  \begin{corollary}\label{cor:anCal}
Let $\alpha$, $\Omega$ be as in \eqref{Omega_def}.  Then the  map from $\mathcal{A}^{1,\alpha}_{\partial \Omega}$ to  $\mathcal{L}(C^{1,\alpha}(\partial\Omega) \times C^{0,\alpha}(\partial\Omega),C^{1,\alpha}(\partial\Omega) \times C^{0,\alpha}(\partial\Omega))$ that takes $\phi$ to the bounded linear operator $\mathcal{C}_\phi$ is real analytic.
\end{corollary}

 We can now see that the map $\phi\mapsto{\cal C}_\phi$ has a holomorphic extension. Indeed, it is well known that for a real vector space $X$ we can consider the complexified vector space $\widetilde X=X+iX$ with the operations 
\[
(x+iy)+(u+iv)=(x+u)+i(y+v) 
\]
and
\[
(a+ib)(x+iy)=(ax-by)+i(bx+ay)
\]
for all $x,y,u,v\in X$ and $a,b\in\mathbb{R}$. If in addition $X$ is a normed space, with norm denoted by $\|\cdot\|_X$, then we might want to equip $\widetilde X$ with a norm as well. How to define a norm on $\widetilde X$ is not, however, a trivial task. We can see, for example, that the function
\[
n(x+iy)=\sqrt{\|x\|_X^2+\|y\|_X^2}
\]
is a norm on $\widetilde X$ only when the norm of $X$ comes from an inner product (we can verify that if $n(\cdot)$ is positive homogeneous, then $\|\cdot\|_X$ has the parallelogram property). In \cite{Ta43} Taylor   proposed to consider the function 
\begin{equation}\label{complexnorm}
\|x+iy\|_{\widetilde X}=\sup_{\Phi\in\overline{\mathbb{B}}_{X^*}}\sqrt{\Phi(x)^2+\Phi(y)^2},
\end{equation}
where $\overline{\mathbb{B}}_{X^*}$ is the closed unit ball in $X^*\equiv {\cal L}(X,\mathbb{R})$ (see also Michal and Wyman \cite{MiWy41}).  We can verify that $\|\cdot\|_{\widetilde X}$ is a norm on $\tilde X$ and that $\widetilde X$ with the norm $\|\cdot\|_{\widetilde X}$ is complete as soon as $X$ is complete.  We can also see that the norm in \eqref{complexnorm} can be written as 
\[
\|x+iy\|_{\widetilde X}=\sup_{t\in[0,2\pi]}\left\|(\cos t)x+(\sin t)y\right\|_X
\]
(cf.~Mu\~nos, Sarantopoulos, and Tonge \cite[Eq.~(1)]{MuEtAl99}). In addition,  every {\em reasonable} norm $\|\cdot\|'_{\widetilde X}$ on $\widetilde X$ that satisfies the conditions  
\[
\|x\|'_{\widetilde X}=\|x\|_X\quad \forall x\in X
\]
and 
\[
\|x+iy\|'_{\widetilde X}=\|x-iy\|'_{\widetilde X}\quad  \forall x,y\in X
\]
is equivalent to $\|\cdot\|_{\widetilde X}$ (cf.~Mu\~nos, Sarantopoulos, and Tonge \cite[Prop.~3]{MuEtAl99}). 

For what concerns this paper, we deduce that $C^{1,\alpha}(\partial\Omega,\mathbb{C})$ (the space of $C^{1,\alpha}$ complex valued functions on $\partial\Omega$) coincides algebraically with the complexified space $\widetilde{C^{1,\alpha}(\partial\Omega)}$ and the standard norm on  $C^{1,\alpha}(\partial\Omega,\mathbb{C})$, which is {\em reasonable} in the  sense introduced above, is equivalent to the norm defined by \eqref{complexnorm}.  Similarly, we have 
\[
\widetilde{C^{1,\alpha}(\partial\Omega,\mathbb{R}^n)}=C^{1,\alpha}(\partial\Omega,\mathbb{C}^n)
\]
algebraically and with equivalent norms, and the complexification of  the real Banach space $\mathcal{L}(C^{1,\alpha}(\partial\Omega) \times C^{0,\alpha}(\partial\Omega),C^{1,\alpha}(\partial\Omega) \times C^{0,\alpha}(\partial\Omega))$ coincides algebraically with 
\[
\mathcal{L}(C^{1,\alpha}(\partial\Omega,\mathbb{C}) \times C^{0,\alpha}(\partial\Omega,\mathbb{C}),C^{1,\alpha}(\partial\Omega,\mathbb{C}) \times C^{0,\alpha}(\partial\Omega,\mathbb{C}))
\] 
and has an equivalent norm.

Then, from Corollary \ref{cor:anCal} and from   H\'ajek and Johanis \cite[Thm. 171, p. 75]{HaJo14} (see also Bochnak \cite[ Thm. 5]{Bo70}) we readily deduce the following.

\begin{corollary} Let $\alpha$, $\Omega$ be as in \eqref{Omega_def}. There exist an open subset $\widetilde{\mathcal{A}}^{1,\alpha}_{\partial \Omega}$ of $C^{1,\alpha}(\partial\Omega,\mathbb{C}^n)$ such that ${\mathcal{A}}^{1,\alpha}_{\partial \Omega}=\widetilde{\mathcal{A}}^{1,\alpha}_{\partial \Omega}\cap C^{1,\alpha}(\partial\Omega,\mathbb{R}^n)$ (that is, ${\mathcal{A}}^{1,\alpha}_{\partial \Omega}$ is the subset of the real valued functions of $\widetilde{\mathcal{A}}^{1,\alpha}_{\partial \Omega}$) and a holomorphic map $\widetilde{\cal C}$ from $\widetilde{\mathcal{A}}^{1,\alpha}_{\partial \Omega}$ to 
\[
\mathcal{L}(C^{1,\alpha}(\partial\Omega,\mathbb{C}) \times C^{0,\alpha}(\partial\Omega,\mathbb{C}),C^{1,\alpha}(\partial\Omega,\mathbb{C}) \times C^{0,\alpha}(\partial\Omega,\mathbb{C}))
\]
such that $\widetilde{\cal C}[\phi]=\mathcal{C}_\phi$ for all $\phi\in {\mathcal{A}}^{1,\alpha}_{\partial \Omega}$.
\end{corollary}


\section{Singular perturbations}\label{s:sin}

In this section we consider the effect of a singular perturbation produced by a small perforation in the domain that is bounded by the support of integration.

We fix
\begin{equation}\label{eq:ass:sing}
	\begin{split}
		&\mbox{$\alpha \in \mathopen]0,1[$ and two bounded open connected subsets $\Omega^o$, $\Omega^i$   of $\mathbb{R}^n$ of class $C^{1,\alpha}$,} 
		\\
		&\mbox{such that their exteriors  $\mathbb{R}^n\setminus \overline{\Omega^o}$ and $\mathbb{R}^n\setminus \overline{\Omega^i}$ are connected,}
		\\
		&\mbox{and the origin  $0$ of $\mathbb{R}^n$ belongs both to $\Omega^o$ and to $\Omega^i$.}
	\end{split}
\end{equation}
Here the superscript  ``$o$'' stands for ``outer domain'' and the superscript ``$i$'' stands for ``inner domain.'' We take
\begin{equation}\label{eq:e0}
\epsilon_0 \equiv \mbox{sup}\{\theta \in \mathopen]0, +\infty\mathclose[: \epsilon \overline{\Omega^i} \subseteq \Omega^o, \ \forall \epsilon \in \mathopen]- \theta, \theta[  \}, 
\end{equation}
and we define the perforated domain $\Omega(\epsilon)$ by setting
\[
\Omega(\epsilon) \equiv \Omega^o \setminus \epsilon \overline{\Omega^i}
\]
for all $\epsilon\in\mathopen]-\epsilon_0,\epsilon_0[$.  Clearly, when $\epsilon$ tends to zero, the set $\Omega(\epsilon)$
degenerates to the punctured domain $\Omega^o \setminus \{0\}$ (see Figure \ref{fig:Omegaeps}).

\begin{figure}[htb]
\centering
\includegraphics[width=4.8in]{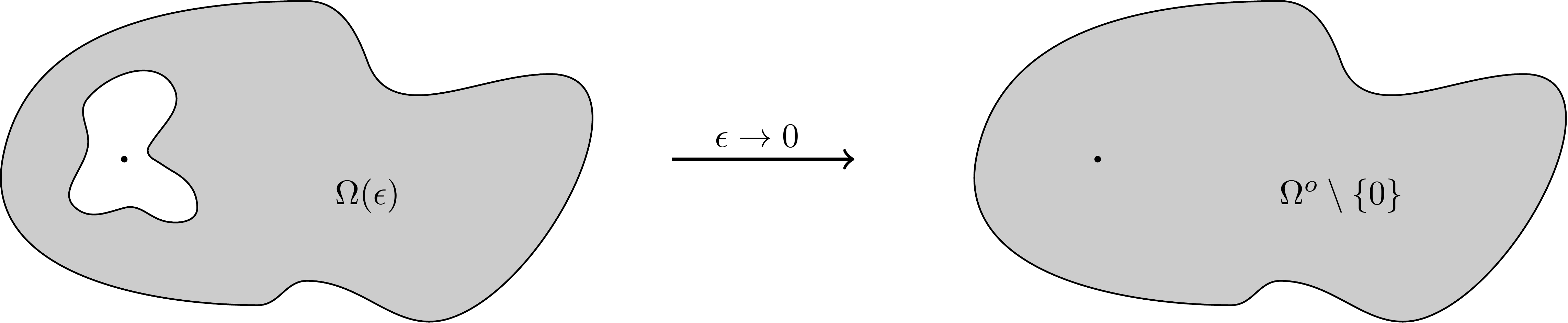}\\
\begin{center}
\caption{{\it The perforated set $\Omega(\epsilon)$ and the limiting punctured set  $\Omega^o\setminus\{0\}$.}} 
		\label{fig:Omegaeps}
\end{center}
\end{figure}

\subsection{The operator $\mathcal{V}_{\Omega(\epsilon)}$}

Our aim is to study the maps that take $\epsilon\in\mathopen]-\epsilon_0,\epsilon_0\mathclose[\setminus\mathopen\{0\mathclose\}$ to the operators $\mathcal{V}_{\Omega(\epsilon)}$, $\mathcal{K}_{\Omega(\epsilon)}$, $\mathcal{K}'_{\Omega(\epsilon)}$, and $\mathcal{W}_{\Omega(\epsilon)}$. We see, however, that these operators are defined on spaces that depend on the parameter $\epsilon$. For example, for every fixed $\epsilon\in\mathopen]-\epsilon_0,\epsilon_0\mathclose[\setminus\{0\}$ the operator $\mathcal{V}_{\Omega(\epsilon)}$ is an element of $\mathcal{L}(C^{0,\alpha}(\partial\Omega(\epsilon)),C^{1,\alpha}(\partial\Omega(\epsilon)))$ (we remind that $\mathcal{V}_{\Omega(\epsilon)}$ is the restriction of the single layer to the boundary of $\Omega(\epsilon)$, see definition \eqref{Vdef}). Then, to describe the dependence of $\mathcal{V}_{\Omega(\epsilon)}$ upon $\epsilon$ we ``pull-back'' the operator to the boundary of the fixed domains $\partial\Omega^o$ and $\partial\Omega^i$. That is,   we define
\begin{align*}
&\mathcal{V}^o_{\epsilon}[\theta^o,\theta^i](x)\equiv\mathcal{V}_{{\Omega(\epsilon)}}[\mu_\epsilon](x)&\forall x\in\partial\Omega^o\,,\\
&\mathcal{V}^i_{\epsilon}[\theta^o,\theta^i](t)\equiv\mathcal{V}_{{\Omega(\epsilon)}}[\mu_\epsilon](\epsilon t)&\forall t\in\partial\Omega^i\,,
\end{align*}
with
\begin{equation}\label{mueps}
\mu_\epsilon (x)\equiv
\left\{
\begin{array}{ll}
\theta^o(x) & \text{if $x\in \partial \Omega^o$}\, ,\\
\theta^i(x/\epsilon) & \text{if $x\in \partial (\epsilon \Omega^i)$}\,,
\end{array}
\right.
\end{equation}
for all $(\theta^o,\theta^i)\in C^{0,\alpha}(\partial \Omega^o)\times C^{0,\alpha}(\partial \Omega^i)$. So, in a sense, we identify functions of  $C^{0,\alpha}(\partial\Omega(\epsilon))$ and $C^{1,\alpha}(\partial\Omega(\epsilon))$ with  elements in the product spaces $C^{0,\alpha}(\partial \Omega^o)\times C^{0,\alpha}(\partial \Omega^i)$ and $C^{1,\alpha}(\partial \Omega^o)\times C^{1,\alpha}(\partial \Omega^i)$, respectively.
Then we set 
\[
{\mathcal{V}}_{\epsilon}\equiv(\mathcal{V}^o_{\epsilon},\mathcal{V}^i_{\epsilon})
\]
and we observe that, for every $\epsilon\in\mathopen]-\epsilon_0,\epsilon_0\mathclose[\setminus \{0\}$, the operator $\mathcal{V}_{\epsilon}$ is an element of a space that does not depend on $\epsilon$, namely 
\[
{\mathcal{V}}_{\epsilon}\in \mathcal{L}(C^{0,\alpha}(\partial \Omega^o)\times C^{0,\alpha}(\partial \Omega^i),C^{1,\alpha}(\partial \Omega^o)\times C^{1,\alpha}(\partial \Omega^i))\,.
\] 
In the following Theorem \ref{Veps} we describe $\mathcal{V}_{\epsilon}$ as a matrix operator with entries written in terms of analytic maps and elementary functions of $\epsilon$.

In what follows we will often use the equality
\begin{equation}\label{der.eq0}
\partial_\epsilon^k(F(\epsilon x))=\sum_{\substack{\beta \in \mathbb{N}^n\\|\beta|=k}}\frac{k!}{\beta!}x^\beta(D^{\beta}F)(\epsilon x),
\end{equation} 
which holds for all $k\in\mathbb{N}$, $\epsilon\in\mathbb{R}$, $x\in\mathbb{R}^n$, and for all functions $F$ analytic in a neighborhood of $\epsilon x$. Here, if $\beta\in\mathbb{N}^n$, then $(D^\beta F)(y)$ denotes the partial derivative of multi-index $\beta$ of the function $F$ evaluated at $y\in\mathbb{R}^n$.

\begin{theorem}\label{Veps} Let $\alpha$, $\Omega^o$, $\Omega^i$ be as in \eqref{eq:ass:sing}. Let $\epsilon_0$ be as in \eqref{eq:e0}. There exist real analytic maps
\[
\begin{aligned}
\mathopen]-\epsilon_0,\epsilon_0[&\to &&\mathcal{L}(C^{0,\alpha}(\partial \Omega^i),C^{1,\alpha}(\partial \Omega^o))\\
\epsilon&\mapsto&&\mathcal{V}^{o,i}_\epsilon
\end{aligned}
\] 
and
\[
\begin{aligned}
\mathopen]-\epsilon_0,\epsilon_0[&\to &&\mathcal{L}(C^{0,\alpha}(\partial \Omega^o),C^{1,\alpha}(\partial \Omega^i))\\
\epsilon&\mapsto&&\mathcal{V}^{i,o}_\epsilon
\end{aligned}
\] 
such that 
\begin{equation}\label{Veps.eq1}
\mathcal{V}_\epsilon=
\left(\begin{array}{cc}
\mathcal{V}_{\Omega^o}&|\epsilon|^{n-1}\mathcal{V}^{o,i}_\epsilon\\
\mathcal{V}^{i,o}_\epsilon&|\epsilon|\,\mathcal{V}_{\Omega^i}-\delta_{2,n}\frac{|\epsilon|\log|\epsilon|}{2\pi}\mathrm{Int}_{\partial\Omega^i}
\end{array}
\right)
\end{equation}
for all $\epsilon\in \mathopen]-\epsilon_0,\epsilon_0[\setminus\{0\}$, where 
\[
\mathrm{Int}_{\partial\Omega^i}[\theta^i]\equiv\int_{\partial\Omega^i}\theta^i\,d\sigma\quad\forall \theta^i\in C^{0,\alpha}(\partial\Omega^i)\,.
\]
Moreover,  the following statements hold.
\begin{itemize}
\item[(i)] The coefficients $\mathcal{V}_{(k)}^{o,i}$ of the power series expansion $\mathcal{V}_\epsilon^{o,i}=\sum_{k=0}^\infty \epsilon^k \mathcal{V}_{(k)}^{o,i}$
with $\epsilon$ in a neighborhood of $0$ are given by
\[
\mathcal{V}_{(k)}^{o,i}[\theta^i](x)= (-1)^{k} \sum_{\substack{\beta \in \mathbb{N}^n\\|\beta|=k}} \frac{1}{\beta!}(D^\beta G_n)(x) \int_{\partial \Omega^i}s^\beta \theta^i(s)\, d\sigma_s 
\]
for all $k\in\mathbb{N}$, $x \in \partial \Omega^o$, and $\theta^i \in C^{0,\alpha}(\partial \Omega^i)$.
\item[(ii)] The coefficients $\mathcal{V}_{(k)}^{i,o}$ of the power series expansion
$\mathcal{V}_\epsilon^{i,o}=\sum_{k=0}^\infty \epsilon^k \mathcal{V}_{(k)}^{i,o}$  with $\epsilon$ in a neighborhood of $0$ are given by
\[
\mathcal{V}_{(k)}^{i,o}[\theta^o](t)=(-1)^{k}\sum_{\substack{\beta \in \mathbb{N}^n\\|\beta|=k}} \frac{1}{\beta!}t^\beta\int_{\partial \Omega^o} (D^\beta G_n)(y)\theta^o(y)\, d\sigma_y 
\]
for all $k\in\mathbb{N}$, $t \in \partial \Omega^i$, and  $\theta^o \in C^{0,\alpha}(\partial \Omega^o)$.
\end{itemize}
\end{theorem}
\begin{proof} Let $(\theta^o,\theta^i) \in C^{0,\alpha}(\partial \Omega^o)\times C^{0,\alpha}(\partial \Omega^i)$, $\epsilon\in \mathopen]-\epsilon_0,\epsilon_0[\setminus\{0\}$. By a computation based on the theorem of change of variable in integrals we have

\[
\begin{split}
\mathcal{V}^o_{\epsilon}[\theta^o,\theta^i](x)&=\int_{\partial \Omega^o}G_n(x-y)\theta^o(y)\, d\sigma_y+|\epsilon|^{n-1}\int_{\partial \Omega^i}G_n(x-\epsilon s)\theta^i(s)\, d\sigma_s\\
&
=\mathcal{V}_{\Omega^o}[\theta^o](x)+|\epsilon|^{n-1}\int_{\partial \Omega^i}G_n(x-\epsilon s)\theta^i(s)\, d\sigma_s 
\end{split}
\]
for all $x\in\partial\Omega^o$. Similarly, we can compute that
\[
\begin{split}
&\mathcal{V}^i_{\epsilon}[\theta^o,\theta^i](t)\\
&\quad=\int_{\partial \Omega^o}G_n(\epsilon t -y)\theta^o(y)\, d\sigma_y+|\epsilon| \int_{\partial \Omega^i}G_n(t-s)\theta^i(s)\, d\sigma_s-\delta_{2,n}\frac{|\epsilon|\log|\epsilon|}{2\pi}\int_{\partial \Omega^i}\theta^i(s)\, d\sigma_s\\
&\quad
=\int_{\partial \Omega^o}G_n(\epsilon t -y)\theta^o(y)\, d\sigma_y+|\epsilon|\;\mathcal{V}_{\Omega^i}[\theta^i](t)-\delta_{2,n}\frac{|\epsilon|\log|\epsilon|}{2\pi}\int_{\partial \Omega^i}\theta^i(s)\, d\sigma_s
\end{split}
\]for all $t\in\partial\Omega^i$, where we have also used the equality
\[
G_n(\epsilon \xi)=|\epsilon|^{2-n}G_n(\xi)-\delta_{2,n}\frac{1}{2\pi}\log|\epsilon|\qquad  \forall \xi \in \mathbb{R}^n \setminus \{0\}\, ,\forall \epsilon \neq 0\, .
\] 
Then equality \eqref{Veps.eq1} holds with
\[
\mathcal{V}^{o,i}_{\epsilon}[\theta^i](x)\equiv\int_{\partial \Omega^i}G_n(x-\epsilon s)\theta^i(s)\, d\sigma_s\qquad\forall \theta^i\in C^{0,\alpha}(\partial\Omega^i)\,,\;\forall x\in\partial\Omega^o
\]
and 
\[
\mathcal{V}^{i,o}_{\epsilon}[\theta^o](t)\equiv\int_{\partial \Omega^o}G_n(\epsilon t -y)\theta^o(y)\, d\sigma_y\qquad\forall \theta^o\in C^{0,\alpha}(\partial\Omega^o)\,,\;\forall t\in\partial\Omega^i\,.
\]
By the regularity results for the integral operators with real analytic kernel of \cite{LaMu13} and by the same argument we have used in the proof of Theorem \ref{thm:anlpop}, we can see that the maps $\epsilon\mapsto \mathcal{V}^{o,i}_\epsilon$ and $\epsilon\mapsto \mathcal{V}^{i,o}_\epsilon$ are real analytic from $\mathopen]-\epsilon_0,\epsilon_0[$ to $\mathcal{L}(C^{0,\alpha}(\partial \Omega^i),C^{1,\alpha}(\partial \Omega^o))$ and from $\mathopen]-\epsilon_0,\epsilon_0[$ to $\mathcal{L}(C^{0,\alpha}(\partial \Omega^o),C^{1,\alpha}(\partial \Omega^i))$, respectively. 

Then we can locally express $\epsilon\mapsto\mathcal{V}^{o,i}_\epsilon$ with its Taylor series. In particular,  we have
\[
\mathcal{V}^{o,i}_\epsilon=\sum_{k=0}^\infty \epsilon^k\frac{1}{k!}(\partial_\epsilon^k\mathcal{V}^{o,i}_\epsilon)_{|\epsilon=0}
\]
for $\epsilon$ in a neighborhood of $0$ and we can prove statement (i) computing the derivatives $(\partial_\epsilon^k\mathcal{V}^{o,i}_\epsilon)_{|\epsilon=0}$. With the help of equation \eqref{der.eq0} we can see that 
\[
\partial_\epsilon^k\bigg(\int_{\partial \Omega^i}G_n(x-\epsilon s)\theta^i(s)\, d\sigma_s\bigg)=(-1)^k \sum_{\substack{\beta \in \mathbb{N}^n\\|\beta|=k}} \frac{k!}{\beta!}\int_{\partial \Omega^i}s^\beta (D^\beta G_n)(x-\epsilon s)\theta^i(s)\, d\sigma_s\,.
\]
Accordingly
\[
\partial_\epsilon^k\bigg(\int_{\partial \Omega^i}G_n(x-\epsilon s)\theta^i(s)\, d\sigma_s\bigg)_{|\epsilon=0}=(-1)^{k} \sum_{\substack{\beta \in \mathbb{N}^n\\|\beta|=k}} \frac{k!}{\beta!}(D^\beta G_n)(x) \int_{\partial \Omega^i}s^\beta \theta^i(s)\, d\sigma_s\,,
\] 
and statement (i) follows. 

Similarly, to verify statement (ii) we have to compute the derivatives $(\partial_\epsilon^k\mathcal{V}^{i,o}_\epsilon)_{|\epsilon=0}$. Again, with the help of \eqref{der.eq0} we see that 
\[
\partial_\epsilon^k\bigg(\int_{\partial \Omega^o}G_n(\epsilon t-y)\theta^o(y)\, d\sigma_y\bigg)= \sum_{\substack{\beta \in \mathbb{N}^n\\|\beta|=k}} \frac{k!}{\beta!}\int_{\partial \Omega^o}t^\beta (D^\beta G_n)(\epsilon t-y)\theta^o(y)\, d\sigma_y\, ,
\]
accordingly
\[
\begin{split}
\partial_\epsilon^k\bigg(\int_{\partial \Omega^o}G_n(\epsilon t-y)\theta^o(y)\, d\sigma_y\bigg)_{|\epsilon=0}&=\sum_{\substack{\beta \in \mathbb{N}^n\\|\beta|=k}} \frac{k!}{\beta!}t^\beta\int_{\partial \Omega^o} (D^\beta G_n)(-y)\theta^o(y)\, d\sigma_y\\ 
&=(-1)^{k}\sum_{\substack{\beta \in \mathbb{N}^n\\|\beta|=k}} \frac{k!}{\beta!}t^\beta\int_{\partial \Omega^o} (D^\beta G_n)(y)\theta^o(y)\, d\sigma_y \, ,
\end{split}
\]
and statement (ii) follows.
\end{proof}

\subsection{The operator $\mathcal{K}_{\Omega(\epsilon)}$}

We proceed with the boundary operator $\mathcal{K}_{\Omega(\epsilon)}$, which is the restriction of the double layer potential to the boundary of $\Omega(\epsilon)$  (see definition 
\eqref{Kdef}).
 In a way that resemble what we did above for the single layer potential, we set
\begin{align*}
&\mathcal{K}^o_{\epsilon}[\theta^o,\theta^i](x)\equiv\mathcal{K}_{{\Omega(\epsilon)}}[\psi_\epsilon](x)&\forall x\in\partial\Omega^o\,,\\
&\mathcal{K}^i_{\epsilon}[\theta^o,\theta^i](t)\equiv\mathcal{K}_{{\Omega(\epsilon)}}[\psi_\epsilon](\epsilon t)&\forall t\in\partial\Omega^i\,,
\end{align*}
for all $\epsilon \in \mathopen]-\epsilon_0,\epsilon_0\mathclose[\setminus \{0\}$ and  $(\theta^o,\theta^i)\in C^{1,\alpha}(\partial \Omega^o)\times C^{1,\alpha}(\partial \Omega^i)$, where
\begin{equation}\label{psieps}
\psi_\epsilon (x)\equiv
\left\{
\begin{array}{ll}
\theta^o(x) & \text{if $x\in \partial \Omega^o$}\, ,\\
\theta^i(x/\epsilon) & \text{if $x\in \partial (\epsilon \Omega^i)$}\, .
\end{array}
\right.
\end{equation}
Then we denote by $\mathcal{K}_\epsilon$ the element of $\mathcal{L}(C^{1,\alpha}(\partial \Omega^o)\times C^{1,\alpha}(\partial \Omega^i),C^{1,\alpha}(\partial \Omega^o)\times C^{1,\alpha}(\partial \Omega^i))$ defined by 
\[
\mathcal{K}_\epsilon\equiv(\mathcal{K}^o_{\epsilon},\mathcal{K}^i_{\epsilon})\quad\forall \epsilon \in \mathopen]-\epsilon_0,\epsilon_0\mathclose[\setminus \{0\}\,.
\]
We have the following.

\begin{theorem}\label{Keps} Let $\alpha$, $\Omega^o$, $\Omega^i$ be as in \eqref{eq:ass:sing}. Let $\epsilon_0$ be as in \eqref{eq:e0}. There exist real analytic maps
\[
\begin{aligned}
 \mathopen]-\epsilon_0,\epsilon_0[&\to &&\mathcal{L}(C^{1,\alpha}(\partial \Omega^i),C^{1,\alpha}(\partial \Omega^o))\\
\epsilon&\mapsto&&\mathcal{K}^{o,i}_\epsilon
\end{aligned}
\] 
and
\[
\begin{aligned}
 \mathopen]-\epsilon_0,\epsilon_0[&\to &&\mathcal{L}(C^{1,\alpha}(\partial \Omega^o),C^{1,\alpha}(\partial \Omega^i))\\
\epsilon&\mapsto&&\mathcal{K}^{i,o}_\epsilon
\end{aligned}
\] 
such that 
\begin{equation}\label{Keps.eq1}
\mathcal{K}_\epsilon=
\left(\begin{array}{cc}
\mathcal{K}_{\Omega^o}&\epsilon|\epsilon|^{n-2}\mathcal{K}^{o,i}_\epsilon\\
\mathcal{K}^{i,o}_\epsilon&-\mathcal{K}_{\Omega^i}
\end{array}
\right)
\end{equation}
for all $\epsilon\in \mathopen]-\epsilon_0,\epsilon_0[\setminus\{0\}$. Moreover,  the following statements hold.
\begin{itemize}
\item[(i)] The coefficients $\mathcal{K}_{(k)}^{o,i}$ of the power series expansion $\mathcal{K}_\epsilon^{o,i}=\sum_{k=0}^\infty \epsilon^k \mathcal{K}_{(k)}^{o,i}$
with $\epsilon$ in a neighborhood of $0$ are given by
\[
\mathcal{K}_{(k)}^{o,i}[\theta^i](x)\equiv  (-1)^{k} \sum_{\substack{\beta \in \mathbb{N}^n\\|\beta|=k}} \frac{1}{\beta!}(\nabla D^\beta G_n)(x) \cdot \int_{\partial \Omega^i}\nu_{\Omega^i}(s)  s^\beta \theta^i(s)\, d\sigma_s 
\]
for all $k\in\mathbb{N}$, $x \in \partial \Omega^o$, and $\theta^i \in C^{1,\alpha}(\partial \Omega^i)$.
\item[(ii)] The coefficients $\mathcal{K}_{(k)}^{i,o}$ of the power series expansion
$\mathcal{K}_\epsilon^{i,o}=\sum_{k=0}^\infty \epsilon^k \mathcal{K}_{(k)}^{i,o}$  with $\epsilon$ in a neighborhood of $0$ are given by
\[
\mathcal{K}_{(k)}^{i,o}[\theta^o](t)\equiv (-1)^{k}\sum_{\substack{\beta \in \mathbb{N}^n\\|\beta|=k}} \frac{1}{\beta!} t^\beta \int_{\partial \Omega^o}\nu_{\Omega^o}(y) \cdot  (\nabla D^\beta G_n)(y)\theta^o(y)\, d\sigma_y 
\]
for all $k\in\mathbb{N}$, $t \in \partial \Omega^i$, and  $\theta^o \in C^{1,\alpha}(\partial \Omega^o)$.
\end{itemize}
\end{theorem} 
\begin{proof} Let $(\theta^o,\theta^i) \in C^{1,\alpha}(\partial \Omega^o)\times C^{1,\alpha}(\partial \Omega^i)$, $\epsilon\in \mathopen]-\epsilon_0,\epsilon_0[\setminus\{0\}$.  By the theorem of change of variables in integrals we can see that
\[
\begin{split}
 \mathcal{K}^o_{\epsilon}[\theta^o,\theta^i](x)&=\mathcal{K}_{\Omega^o}[\theta^o](x)+|\epsilon|^{n-1}\mathrm{sgn}(\epsilon) \int_{\partial \Omega^i}\nu_{\Omega^i}(s)\cdot \nabla G_n(x-\epsilon s)\theta^i(s)\, d\sigma_s\\
&
=\mathcal{K}_{\Omega^o}[\theta^o](x)+\epsilon |\epsilon|^{n-2}\int_{\partial \Omega^i}\nu_{\Omega^i}(s)\cdot \nabla G_n(x-\epsilon s)\theta^i(s)\, d\sigma_s
\end{split}
\]
for all $x\in\partial\Omega^o$. Moreover, by equality
\[
\nabla G_n(\epsilon \eta) =\mathrm{sgn}(\epsilon)|\epsilon|^{1-n}\nabla G_n( \eta ) \qquad \forall \epsilon \in  \mathbb{R}\setminus \{0\}, \, \forall \eta \in \mathbb{R}^n \setminus \{0\},
\]
we can compute that  
\[
 \mathcal{K}^i_{\epsilon}[\theta^o,\theta^i](t)=-\int_{\partial \Omega^o}\nu_{\Omega^o}(y)\cdot \nabla G_n(\epsilon t -y)\theta^o(y)\, d\sigma_y- \mathcal{K}_{\Omega^i}[\theta^i](t)  \qquad \forall t \in \partial\Omega^i\, .
\]
Then equality \eqref{Keps.eq1} holds with
\[
\mathcal{K}^{o,i}_{\epsilon}[\theta^i](x)\equiv\int_{\partial \Omega^i}\nu_{\Omega^i}(s)\cdot \nabla G_n(x-\epsilon s)\theta^i(s)\, d\sigma_s\qquad\forall \theta^i\in C^{1,\alpha}(\partial\Omega^i)\,,\;\forall x\in\partial\Omega^o
\]
and 
\[
\mathcal{K}^{i,o}_{\epsilon}[\theta^o](t)\equiv-\int_{\partial \Omega^o}\nu_{\Omega^o}(y)\cdot \nabla G_n(\epsilon t -y)\theta^o(y)\, d\sigma_y\qquad\forall \theta^o\in C^{1,\alpha}(\partial\Omega^o)\,,\;\forall t\in\partial\Omega^i\,.
\]
By the regularity results for the integral operators with real analytic kernel of \cite{LaMu13} and by the same argument we have used in the proof of Theorem \ref{thm:anlpop}, we can see that the maps $\epsilon\mapsto \mathcal{K}^{o,i}_\epsilon$ and $\epsilon\mapsto \mathcal{K}^{i,o}_\epsilon$ are real analytic from $\mathopen]-\epsilon_0,\epsilon_0[$ to $\mathcal{L}(C^{1,\alpha}(\partial \Omega^i),C^{1,\alpha}(\partial \Omega^o))$ and from $\mathopen]-\epsilon_0,\epsilon_0[$ to $\mathcal{L}(C^{1,\alpha}(\partial \Omega^o),C^{1,\alpha}(\partial \Omega^i))$, respectively. 

Then we can locally express $\epsilon\mapsto\mathcal{K}^{o,i}_\epsilon$ with its Taylor series. In particular,  we have
\[
\mathcal{K}^{o,i}_\epsilon=\sum_{k=0}^\infty \epsilon^k\frac{1}{k!}(\partial_\epsilon^k\mathcal{K}^{o,i}_\epsilon)_{|\epsilon=0}
\]
for $\epsilon$ in a neighborhood of $0$ and we can prove statement (i) computing the derivatives $(\partial_\epsilon^k\mathcal{K}^{o,i}_\epsilon)_{|\epsilon=0}$. With the help of equation \eqref{der.eq0} we can see that 
\[
\begin{split}
&\partial_\epsilon^k\bigg( \int_{\partial \Omega^i}\nu_{\Omega^i}(s)\cdot \nabla G_n(x-\epsilon s)\theta^i(s)\, d\sigma_s\bigg)\\
&\qquad=(-1)^k \sum_{\substack{\beta \in \mathbb{N}^n\\|\beta|=k}} \frac{k!}{\beta!}\int_{\partial \Omega^i}\nu_{\Omega^i}(s) \cdot  (\nabla D^\beta G_n)(x-\epsilon s) s^\beta \theta^i(s)\, d\sigma_s\, ,
\end{split}
\]
and accordingly
\[
\begin{split}
&\partial_\epsilon^k\bigg( \int_{\partial \Omega^i}\nu_{\Omega^i}(s)\cdot \nabla G_n(x-\epsilon s)\theta^i(s)\, d\sigma_s\bigg)_{|\epsilon=0}\\
&\qquad=(-1)^{k} \sum_{\substack{\beta \in \mathbb{N}^n\\|\beta|=k}} \frac{k!}{\beta!}(\nabla D^\beta G_n)(x) \cdot \int_{\partial \Omega^i}\nu_{\Omega^i}(s)  s^\beta \theta^i(s)\, d\sigma_s \, .
\end{split}
\]
and statement (i) follows. 

Similarly, to verify statement (ii) we have to compute the derivatives $(\partial_\epsilon^k\mathcal{K}^{i,o}_\epsilon)_{|\epsilon=0}$. Again, with the help of \eqref{der.eq0} we see that 
\[
\begin{split}
&\partial_\epsilon^k\bigg( \int_{\partial \Omega^o}\nu_{\Omega^o}(y)\cdot \nabla G_n(\epsilon t -y)\theta^o(y)\, d\sigma_y\bigg)\\
&\qquad=\sum_{\substack{\beta \in \mathbb{N}^n\\|\beta|=k}} \frac{k!}{\beta!}\int_{\partial \Omega^o}\nu_{\Omega^o}(y) \cdot  (\nabla D^\beta G_n)(\epsilon t-y) t^\beta \theta^o(y)\, d\sigma_y\, ,
\end{split}
\]
and accordingly
\[
\begin{split}
&\partial_\epsilon^k\bigg(-\int_{\partial \Omega^o}\nu_{\Omega^o}(y)\cdot \nabla G_n(\epsilon t -y)\theta^o(y)\, d\sigma_y\bigg)_{|\epsilon=0}\\
&\qquad=-\sum_{\substack{\beta \in \mathbb{N}^n\\|\beta|=k}} \frac{k!}{\beta!}\int_{\partial \Omega^o}\nu_{\Omega^o}(y) \cdot  (\nabla D^\beta G_n)(-y) t^\beta \theta^o(y)\, d\sigma_y\\
&\qquad=(-1)^{k}\sum_{\substack{\beta \in \mathbb{N}^n\\|\beta|=k}} \frac{k!}{\beta!}t^\beta \int_{\partial \Omega^i}\nu_{\Omega^o}(y) \cdot  (\nabla D^\beta G_n)(y)\theta^o(y)\, d\sigma_y \, .
\end{split}
\]
and statement (ii) follows.
\end{proof}


\subsection{The operator $\mathcal{K}'_{\Omega(\epsilon)}$}

We now turn to $\mathcal{K}'_{\Omega(\epsilon)}$,  the boundary operator related with the normal derivative of the single layer potential (see definition \eqref{K'def}). We set
\begin{align*}
&\mathcal{K}'^o_{\epsilon}[\theta^o,\theta^i](x)\equiv\mathcal{K}'_{{\Omega(\epsilon)}}[\mu_\epsilon](x)&\forall x\in\partial\Omega^o\,,\\
&\mathcal{K}'^i_{\epsilon}[\theta^o,\theta^i](t)\equiv\mathcal{K}'_{{\Omega(\epsilon)}}[\mu_\epsilon](\epsilon t)&\forall t\in\partial\Omega^i\,,
\end{align*}
for all $\epsilon \in \mathopen]-\epsilon_0,\epsilon_0\mathclose[\setminus \{0\}$ and  $(\theta^o,\theta^i)\in C^{0,\alpha}(\partial \Omega^o)\times C^{0,\alpha}(\partial \Omega^i)$, with $\mu_\epsilon$ as in \eqref{mueps}. Then we define
\[
\mathcal{K}'_\epsilon\equiv(\mathcal{K}'^o_{\epsilon},\mathcal{K}'^i_{\epsilon})\quad\forall \epsilon \in \mathopen]-\epsilon_0,\epsilon_0\mathclose[\setminus \{0\}
\]
and we note that $\mathcal{K}'_\epsilon$ is an element of $\mathcal{L}(C^{0,\alpha}(\partial \Omega^o)\times C^{0,\alpha}(\partial \Omega^i),C^{0,\alpha}(\partial \Omega^o)\times C^{0,\alpha}(\partial \Omega^i))$.
We can prove the following.

\begin{theorem}\label{K'eps} Let $\alpha$, $\Omega^o$, $\Omega^i$ be as in \eqref{eq:ass:sing}. Let $\epsilon_0$ be as in \eqref{eq:e0}. There exist real analytic maps
\[
\begin{aligned}
\mathopen]-\epsilon_0,\epsilon_0[&\to &&\mathcal{L}(C^{0,\alpha}(\partial \Omega^i),C^{0,\alpha}(\partial \Omega^o))\\
\epsilon&\mapsto&&\mathcal{K}'^{o,i}_\epsilon
\end{aligned}
\] 
and
\[
\begin{aligned}
\mathopen]-\epsilon_0,\epsilon_0[&\to &&\mathcal{L}(C^{0,\alpha}(\partial \Omega^o),C^{0,\alpha}(\partial \Omega^i))\\
\epsilon&\mapsto&&\mathcal{K}'^{i,o}_\epsilon
\end{aligned}
\] 
such that 
\begin{equation}\label{K'eps.eq1}
\mathcal{K}'_\epsilon=
\left(\begin{array}{cc}
\mathcal{K}'_{\Omega^o}&|\epsilon|^{n-1}\mathcal{K}'^{o,i}_\epsilon\\
\mathrm{sgn}(\epsilon)\mathcal{K}'^{i,o}_\epsilon&-\mathcal{K}'_{\Omega^i}
\end{array}
\right)
\end{equation}
for all $\epsilon\in \mathopen]-\epsilon_0,\epsilon_0[\setminus\{0\}$. Moreover,  the following statements hold.
\begin{itemize}
\item[(i)] The coefficients $\mathcal{K}'^{o,i}_{(k)}$ of the power series expansion $\mathcal{K}'^{o,i}_\epsilon=\sum_{k=0}^\infty \epsilon^k \mathcal{K}'^{o,i}_{(k)}$
with $\epsilon$ in a neighborhood of $0$ are given by
\[
\mathcal{K}_{(k)}^{\prime o,i}[\theta^i](x)\equiv  (-1)^{k} \sum_{\substack{\beta \in \mathbb{N}^n\\|\beta|=k}} \frac{1}{\beta!}\nu_{\Omega^o}(x) \cdot (\nabla D^\beta G_n)(x)\int_{\partial \Omega^i} s^\beta \theta^i(s)\, d\sigma_s 
\]
for all $k\in\mathbb{N}$, $x \in \partial \Omega^o$, and $\theta^i \in C^{0,\alpha}(\partial \Omega^i)$.
\item[(ii)] The coefficients $\mathcal{K}'^{i,o}_{(k)}$ of the power series expansion
$\mathcal{K}'^{i,o}_\epsilon=\sum_{k=0}^\infty \epsilon^k \mathcal{K}'^{i,o}_{(k)}$  with $\epsilon$ in a neighborhood of $0$ are given by
\[
\mathcal{K}_{(k)}^{\prime i,o}[\theta^o](t)\equiv (-1)^{k}\sum_{\substack{\beta \in \mathbb{N}^n\\|\beta|=k}} \frac{1}{\beta!}t^\beta \nu_{\Omega^i}(t) \cdot \int_{\partial \Omega^o} (\nabla D^\beta G_n)(y)\theta^o(y)\, d\sigma_y
\]
for all $k\in\mathbb{N}$, $t \in \partial \Omega^i$, and  $\theta^o \in C^{0,\alpha}(\partial \Omega^o)$.
\end{itemize}
\end{theorem}
\begin{proof}  Let $(\theta^o,\theta^i) \in C^{0,\alpha}(\partial \Omega^o)\times C^{0,\alpha}(\partial \Omega^i)$, $\epsilon\in \mathopen]-\epsilon_0,\epsilon_0[\setminus\{0\}$.  By a straightforward computation based on the theorem of change of variable in integrals we can see that 
\[
\begin{aligned}
\mathcal{K}'^o_{\epsilon}[\theta^o,\theta^i](x)&
=\mathcal{K}'_{\Omega^o}[\theta^o](x)+ |\epsilon|^{n-1}\int_{\partial \Omega^i}\nu_{\Omega^o}(x)\cdot \nabla G_n(x-\epsilon s)\theta^i(s)\, d\sigma_s&\forall x \in \partial \Omega^o\,,\\
\mathcal{K}'^i_{\epsilon}[\theta^o,\theta^i](t)&=-\mathrm{sgn}(\epsilon)\int_{\partial \Omega^o}\nu_{\Omega^i}(t)\cdot \nabla G_n(\epsilon t -y)\theta^o(y)\, d\sigma_y- \mathcal{K}'_{\Omega^i}[\theta^i](t)&\forall t \in \partial \Omega^i\,.
\end{aligned}
\]
Then \eqref{K'eps.eq1} holds with 
\[
\begin{aligned}
\mathcal{K}'^{o,i}_\epsilon[\theta^i](x)&\equiv\int_{\partial \Omega^i}\nu_{\Omega^o}(x)\cdot \nabla G_n(x-\epsilon s)\theta^i(s)\, d\sigma_s&\forall x \in \partial \Omega^o\,,\;\forall\theta^i \in C^{0,\alpha}(\partial \Omega^i)\,,\\
\mathcal{K}_\epsilon^{\prime i,o}[\theta^o](t)&\equiv -\int_{\partial \Omega^o}\nu_{\Omega^i}(t)\cdot \nabla G_n(\epsilon t -y)\theta^o(y)\, d\sigma_y&\forall t \in \partial \Omega^i\,,\;\forall\theta^o \in C^{0,\alpha}(\partial \Omega^o)\,.
\end{aligned}
\]
 By  the regularity results for the integral operators with real analytic kernel of \cite{LaMu13} (see also the argument in the proof of Theorem \ref{thm:anlpop}) we can see that the maps $\epsilon\mapsto \mathcal{K}'^{o,i}_\epsilon$ and $\epsilon\mapsto \mathcal{K}'^{i,o}_\epsilon$ are real analytic from $\mathopen]-\epsilon_0,\epsilon_0[$ to $\mathcal{L}(C^{0,\alpha}(\partial \Omega^i),C^{0,\alpha}(\partial \Omega^o))$ and from $\mathopen]-\epsilon_0,\epsilon_0[$ to $\mathcal{L}(C^{0,\alpha}(\partial \Omega^o),C^{0,\alpha}(\partial \Omega^i))$, respectively. 
 
To verify statement (i) we compute
\[
\begin{split}
&\partial_\epsilon^k\bigg( \int_{\partial \Omega^i}\nu_{\Omega^o}(x)\cdot \nabla G_n(x-\epsilon s)\theta^i(s)\, d\sigma_s\bigg)\\
&\qquad=(-1)^k \sum_{\substack{\beta \in \mathbb{N}^n\\|\beta|=k}} \frac{k!}{\beta!}\int_{\partial \Omega^i}\nu_{\Omega^o}(x) \cdot  (\nabla D^\beta G_n)(x-\epsilon s)s^\beta \theta^i(s)\, d\sigma_s\, ,
\end{split}
\]
and accordingly
\[
\begin{split}
&\partial_\epsilon^k\bigg(\int_{\partial \Omega^i}\nu_{\Omega^o}(x)\cdot \nabla G_n(x-\epsilon s)\theta^i(s)\, d\sigma_s\bigg)_{|\epsilon=0}\\
&\qquad=(-1)^{k} \sum_{\substack{\beta \in \mathbb{N}^n\\|\beta|=k}} \frac{k!}{\beta!}\nu_{\Omega^o}(x) \cdot (\nabla D^\beta G_n)(x)\int_{\partial \Omega^i} s^\beta \theta^i(s)\, d\sigma_s \, .
\end{split}
\]

To verify statement (ii), we note that  we have
\[
\begin{split}
&\partial_\epsilon^k\bigg( \int_{\partial \Omega^o}\nu_{\Omega^i}(t)\cdot \nabla G_n(\epsilon t -y)\theta^o(y)\, d\sigma_y\bigg)\\
&\qquad=\sum_{\substack{\beta \in \mathbb{N}^n\\|\beta|=k}} \frac{k!}{\beta!}\int_{\partial \Omega^o}\nu_{\Omega^i}(t) \cdot  (\nabla D^\beta G_n)(\epsilon t-y) t^\beta \theta^o(y)\, d\sigma_y
\end{split}
\]
and accordingly  
\[
\begin{split}
&\partial_\epsilon^k\bigg(-\int_{\partial \Omega^o}\nu_{\Omega^i}(t)\cdot \nabla G_n(\epsilon t -y)\theta^o(y)\, d\sigma_y\bigg)_{|\epsilon=0}\\
&\qquad=-\sum_{\substack{\beta \in \mathbb{N}^n\\|\beta|=k}} \frac{k!}{\beta!}\int_{\partial \Omega^o}\nu_{\Omega^i}(t) \cdot  (\nabla D^\beta G_n)(-y)t^\beta \theta^o(y)\, d\sigma_y\\
&\qquad=(-1)^{k}\sum_{\substack{\beta \in \mathbb{N}^n\\|\beta|=k}} \frac{k!}{\beta!}t^\beta \nu_{\Omega^i}(t) \cdot \int_{\partial \Omega^o} (\nabla D^\beta G_n)(y)\theta^o(y)\, d\sigma_y \, .
\end{split}
\]
\end{proof}


\subsection{The operator $\mathcal{W}_{\Omega(\epsilon)}$}

The last operator to consider is $\mathcal{W}_{{\Omega(\epsilon)}}$  (see definition \eqref{Wdef}). As usual, we define
\begin{align*}
&\mathcal{W}^o_{\epsilon}[\theta^o,\theta^i](x)\equiv\mathcal{W}_{{\Omega(\epsilon)}}[\psi_\epsilon](x)&\forall x\in\partial\Omega^o\,,\\
&\mathcal{W}^i_{\epsilon}[\theta^o,\theta^i](t)\equiv\mathcal{W}_{{\Omega(\epsilon)}}[\psi_\epsilon](\epsilon t)&\forall t\in\partial\Omega^i\,,
\end{align*}
for all $\epsilon \in \mathopen]-\epsilon_0,\epsilon_0\mathclose[\setminus \{0\}$ and  $(\theta^o,\theta^i)\in C^{1,\alpha}(\partial \Omega^o)\times C^{1,\alpha}(\partial \Omega^i)$, with $\psi_\epsilon$ as in \eqref{psieps}. Then we take 
\[
\mathcal{W}_\epsilon\equiv(\mathcal{W}^o_{\epsilon},\mathcal{W}^i_{\epsilon})\quad\forall \epsilon \in \mathopen]-\epsilon_0,\epsilon_0\mathclose[\setminus \{0\}
\]
and we wish to describe the map $\epsilon\mapsto\mathcal{W}_\epsilon$ from $]-\epsilon_0,\epsilon_0\mathclose[\setminus \{0\}$ to  $\mathcal{L}(C^{1,\alpha}(\partial \Omega^o)\times C^{1,\alpha}(\partial \Omega^i),C^{0,\alpha}(\partial \Omega^o)\times C^{0,\alpha}(\partial \Omega^i))$.

\begin{theorem}\label{Weps} Let $\alpha$, $\Omega^o$, $\Omega^i$ be as in \eqref{eq:ass:sing}. Let $\epsilon_0$ be as in \eqref{eq:e0}. There exist real analytic maps
\[
\begin{aligned}
\mathopen]-\epsilon_0,\epsilon_0[&\to &&\mathcal{L}(C^{1,\alpha}(\partial \Omega^i),C^{0,\alpha}(\partial \Omega^o))\\
\epsilon&\mapsto&&\mathcal{W}^{o,i}_\epsilon
\end{aligned}
\] 
and
\[
\begin{aligned}
\mathopen]-\epsilon_0,\epsilon_0[&\to &&\mathcal{L}(C^{1,\alpha}(\partial \Omega^o),C^{0,\alpha}(\partial \Omega^i))\\
\epsilon&\mapsto&&\mathcal{W}^{i,o}_\epsilon
\end{aligned}
\] 
such that 
\begin{equation}\label{Weps.eq1}
\mathcal{W}_\epsilon=
\left(\begin{array}{cc}
\mathcal{W}_{\Omega^o}&\epsilon|\epsilon|^{n-2}\mathcal{W}^{o,i}_\epsilon\\
|\epsilon|^{-1}\mathcal{W}^{i,o}_\epsilon&|\epsilon|^{-1}\mathcal{W}_{\Omega^i}
\end{array}
\right)
\end{equation}
for all $\epsilon\in \mathopen]-\epsilon_0,\epsilon_0[\setminus\{0\}$. Moreover,  the following statements hold.
\begin{itemize}
\item[(i)] The coefficients $\mathcal{W}^{o,i}_{(k)}$ of the power series expansion $\mathcal{W}^{o,i}_\epsilon=\sum_{k=0}^\infty \epsilon^k \mathcal{W}^{o,i}_{(k)}$
with $\epsilon$ in a neighborhood of $0$ are given by
\[
\mathcal{W}_{(k)}^{o,i}[\theta^i](x)\equiv  (-1)^{k+1} \sum_{\substack{\beta \in \mathbb{N}^n\\|\beta|=k}} \frac{1}{\beta!}\nu_{\Omega^o}(x)\cdot \nabla_x  \Bigg((\nabla D^\beta G_n)(x)\cdot \int_{\partial \Omega^i}\nu_{\Omega^i}(s)  s^\beta \theta^i(s)\, d\sigma_s\Bigg)
\]
for all $k\in\mathbb{N}$, $x \in \partial \Omega^o$, and $\theta^i \in C^{1,\alpha}(\partial \Omega^i)$.
\item[(ii)] The coefficients $\mathcal{W}^{i,o}_{(k)}$ of the power series expansion
$\mathcal{W}^{i,o}_\epsilon=\sum_{k=0}^\infty \epsilon^k \mathcal{W}^{i,o}_{(k)}$  with $\epsilon$ in a neighborhood of $0$ are given by
\[
\mathcal{W}_{(k)}^{i,o}[\theta^o](t)\equiv (-1)^{k}\sum_{\substack{\beta \in \mathbb{N}^n\\|\beta|=k}} \frac{1}{\beta!}\frac{\partial t^\beta}{\partial \nu_{\Omega^i}(t)} \int_{\partial \Omega^o}\nu_{\Omega^o}(y) \cdot  (\nabla D^\beta G_n)(y)\theta^o(y)\, d\sigma_y 
\]
for all $k\in\mathbb{N}$, $t \in \partial \Omega^i$, and  $\theta^o \in C^{1,\alpha}(\partial \Omega^o)$.
\end{itemize}
\end{theorem}
\begin{proof}  Let $(\theta^o,\theta^i) \in C^{1,\alpha}(\partial \Omega^o)\times C^{1,\alpha}(\partial \Omega^i)$, $\epsilon\in \mathopen]-\epsilon_0,\epsilon_0[\setminus\{0\}$.  By the theorem of change of variable in integrals we can compute that 
\[
\begin{split}
\mathcal{W}_{\epsilon}^o[\theta^o,\theta^i](x)&=\mathcal{W}_{\Omega^o}[\theta^o](x)-|\epsilon|^{n-1}\mathrm{sgn}(\epsilon) \nu_{\Omega^o}(x)\cdot \nabla_x  \int_{\partial \Omega^i}\nu_{\Omega^i}(s)\cdot \nabla G_n(x-\epsilon s)\theta^i(s)\, d\sigma_s\\
&
=\mathcal{W}_{\Omega^o}[\theta^o](x)-\epsilon |\epsilon|^{n-2}\nu_{\Omega^o}(x)\cdot \nabla_x\int_{\partial \Omega^i}\nu_{\Omega^i}(s)\cdot \nabla G_n(x-\epsilon s)\theta^i(s)\, d\sigma_s \end{split}
\]
for all $x\in\partial \Omega^o$, and 
\[
\mathcal{W}_{\epsilon}^i[\theta^o,\theta^i](t)=-|\epsilon|^{-1} \nu_{\Omega^i}(t)\cdot \nabla_t \int_{\partial \Omega^o}\nu_{\Omega^o}(y)\cdot \nabla G_n(\epsilon t -y)\theta^o(y)\, d\sigma_y+|\epsilon|^{-1} \mathcal{W}_{\Omega^i}[\theta^i](t)  
\]
for all $t\in \partial \Omega^i$. Then \eqref{Weps.eq1} holds with 
\[
\mathcal{W}^{o,i}_\epsilon[\theta^i](x)\equiv-\nu_{\Omega^o}(x)\cdot \nabla_x\int_{\partial \Omega^i}\nu_{\Omega^i}(s)\cdot \nabla G_n(x-\epsilon s)\theta^i(s)\, d\sigma_s
\]
for all $x\in\partial\Omega^o$ and $\theta^i \in C^{1,\alpha}(\partial \Omega^i)$, and 
\[
\mathcal{W}_\epsilon^{i,o}[\theta^o](t)\equiv -\nu_{\Omega^i}(t)\cdot \nabla_t \int_{\partial \Omega^o}\nu_{\Omega^o}(y)\cdot \nabla G_n(\epsilon t -y)\theta^o(y)\, d\sigma_y  
\]
for all $t \in \partial \Omega^i$, $\theta^o \in C^{1,\alpha}(\partial \Omega^o)$.

 By  the regularity results for the integral operators with real analytic kernel of \cite{LaMu13} (see also the argument in the proof of Theorem \ref{thm:anlpop}) we can verify that the maps $\epsilon\mapsto \mathcal{W}^{o,i}_\epsilon$ and $\epsilon\mapsto \mathcal{W}^{i,o}_\epsilon$ are real analytic from $\mathopen]-\epsilon_0,\epsilon_0[$ to $\mathcal{L}(C^{1,\alpha}(\partial \Omega^i),C^{0,\alpha}(\partial \Omega^o))$ and from $\mathopen]-\epsilon_0,\epsilon_0[$ to $\mathcal{L}(C^{1,\alpha}(\partial \Omega^o),C^{0,\alpha}(\partial \Omega^i))$, respectively. 
 
To verify statement (i) we compute
\[
\begin{split}
&\partial_\epsilon^k\bigg(-\nu_{\Omega^o}(x)\cdot \nabla_x \int_{\partial \Omega^i}\nu_{\Omega^i}(s)\cdot \nabla G_n(x-\epsilon s)\theta^i(s)\, d\sigma_s\bigg)\\
&\qquad=(-1)^{k+1} \sum_{\substack{\beta \in \mathbb{N}^n\\|\beta|=k}} \frac{k!}{\beta!}\nu_{\Omega^o}(x)\cdot \nabla_x\Bigg(\int_{\partial \Omega^i}\nu_{\Omega^i}(s) \cdot  (\nabla D^\beta G_n)(x-\epsilon s)s^\beta \theta^i(s)\, d\sigma_s\Bigg)\, ,
\end{split}
\]
and accordingly
\[
\begin{split}
&\partial_\epsilon^k\bigg(-\nu_{\Omega^o}(x)\cdot \nabla_x \int_{\partial \Omega^i}\nu_{\Omega^i}(s)\cdot \nabla G_n(x-\epsilon s)\theta^i(s)\, d\sigma_s\bigg)_{|\epsilon=0}\\
&\qquad=(-1)^{k+1} \sum_{\substack{\beta \in \mathbb{N}^n\\|\beta|=k}} \frac{k!}{\beta!}\nu_{\Omega^o}(x)\cdot \nabla_x \Bigg((\nabla D^\beta G_n)(x)\cdot \int_{\partial \Omega^i}\nu_{\Omega^i}(s)  s^\beta \theta^i(s)\, d\sigma_s\Bigg) \, .
\end{split}
\]
To verify statement (ii) we compute
\[
\begin{split}
&\partial_\epsilon^k\bigg( -\nu_{\Omega^i}(t)\cdot \nabla_t \int_{\partial \Omega^o}\nu_{\Omega^o}(y)\cdot \nabla G_n(\epsilon t -y)\theta^o(y)\, d\sigma_y\bigg)\\
&\qquad=-\sum_{\substack{\beta \in \mathbb{N}^n\\|\beta|=k}} \frac{k!}{\beta!}\nu_{\Omega^i}(t)\cdot \nabla_t \int_{\partial \Omega^o}\nu_{\Omega^o}(y) \cdot  (\nabla D^\beta G_n)(\epsilon t-y) t^\beta \theta^o(y)\, d\sigma_y
\end{split}
\]
and accordingly
\[
\begin{split}
&\partial_\epsilon^k\bigg(-\nu_{\Omega^i}(t)\cdot \nabla_t \int_{\partial \Omega^o}\nu_{\Omega^o}(y)\cdot \nabla G_n(\epsilon t -y)\theta^o(y)\, d\sigma_y \bigg)_{|\epsilon=0}\\
&\qquad=-\sum_{\substack{\beta \in \mathbb{N}^n\\|\beta|=k}} \frac{k!}{\beta!}\nu_{\Omega^i}(t)\cdot \nabla_t \int_{\partial \Omega^o}\nu_{\Omega^o}(y) \cdot  (\nabla D^\beta G_n)(-y) t^\beta \theta^o(y)\, d\sigma_y\\
&\qquad=(-1)^{k}\sum_{\substack{\beta \in \mathbb{N}^n\\|\beta|=k}} \frac{k!}{\beta!}\frac{\partial t^\beta}{\partial \nu_{\Omega^i}(t)} \int_{\partial \Omega^o}\nu_{\Omega^o}(y) \cdot  (\nabla D^\beta G_n)(y)\theta^o(y)\, d\sigma_y \, .
\end{split}
\]
\end{proof}

In conclusion of this section, we note that, putting together  the results obtained for the operators 
$\mathcal{V}_\epsilon$, $\mathcal{K}_\epsilon$,  $\mathcal{K}'_\epsilon $ and $\mathcal{W}_\epsilon $, we may also describe the map that takes  $\epsilon$ to the (pull-back of) the corresponding Calder\'on projector.

\subsection*{Acknowledgment}

The authors are members of the ``Gruppo Nazionale per l'Analisi Matematica, la Probabilit\`a e le loro Applicazioni'' (GNAMPA) of the ``Istituto Nazionale di Alta Matematica'' (INdAM). P.L.~and P.M.~acknowledge the support of the Project BIRD191739/19 ``Sensitivity analysis of partial differential equations in
the mathematical theory of electromagnetism'' of the University of Padova.  
P.M.~acknowledges the support of  the grant ``Challenges in Asymptotic and Shape Analysis - CASA''  of the Ca' Foscari University of Venice.  P.M.~also acknowledges the support from EU through the H2020-MSCA-RISE-2020 project EffectFact, 
Grant agreement ID: 101008140.

\end{document}